\documentclass[smallcondensed]{svjour3}     
\usepackage{graphicx,amsmath,amssymb,cite}
\usepackage[colorlinks,
citecolor=blue]{hyperref}
\DeclareMathOperator{\tr}{trace}

\DeclareMathOperator{\Span}{span}

\DeclareMathOperator{\diver}{div}

\newcommand{\co}{\nabla}

\newcommand{\di}{\mathcal{D}}
\newcommand{\h}{\mathcal{H}}
\newcommand{\V}{\mathcal{V}}
\newcommand{\T}{\mathcal{T}}
\newcommand{\A}{\mathcal{A}}
\newcommand{\R}{\mathbb{R}}

\begin{document}

\title{On anti-invariant  semi-Riemannian submersions from Lorentzian (para)Sasakian manifolds}

\titlerunning{On anti-invariant  semi-Riemannian submersions from Lorentzian(para)Sasakian}        

\author{Morteza Faghfouri         \and
        Sahar Mashmouli 
}


\institute{M. Faghfouri \at
Faculty of Mathematics,\\ University of Tabriz,\\ Tabriz, Iran.\\
              \email{faghfouri@tabrizu.ac.ir}           
           \and
           S. Mashmouli \at
              Faculty of Mathematics,\\ University of Tabriz,\\ Tabriz, Iran.\\
              \email{s\_mashmouli91@ms.tabrizu.ac.ir}
}

\date{Received: date / Accepted: date}

\maketitle

\begin{abstract}
In this paper we study a semi-Riemannian submersion from Lorentzian (para)almost contact manifolds and find necessary and sufficient conditions for the characteristic vector field to be vertical or horizontal. We also obtain decomposition theorems for an anti-invariant semi-Riemannian submersion from Lorentzian (para)Sasakian manifolds onto a Lorentzian manifold.
\end{abstract}

\section{Introduction}

Semi-Riemannian submersions between semi-Riemannian manifolds
were studied by O'Neill \cite{Oneill:Thefundamentalequationsofsubmersion,oneil:book} and Gray \cite{Gray:PseudoRiemannianAlmostProductManifoldsSubmersions}.
Moreover, B. {\d{S}}ahin  in \cite{Shahin:SlantSubmersionHermitian,shahin:AntiInvariantRiemannianSubmersionsFromAlmostHermitianManifolds} introduced anti-invariant Riemannian submersions and slant submersions from almost Hermitian manifold onto Riemannian manifolds.
Also, anti-invariant Riemannian submersions were studied in \cite{Murathan:AntiInvariantRiemannianSubmersionFromCosymplectic,Erken:SlantRiemanniansubmersionsfromSasakianmanifolds2016,Erken:AntiInvariantRiemanniansubmersionsfromSasakianmanifolds,Lee:AntInvariantRiemannianSubmersionsAlmostContact}.

The theory of Lorentzian submersion was introduced  by  Magin and Falcitelli {\it et al} in \cite{Magid:SubmersionsFromSntiDeSitterspace} and \cite{Falcitelli:RiemannianSubmersionsandRelatedTopics}, respectively.
In \cite{Kaneyuki:Almostparacontact} Kaneyuki and Williams defined the almost
paracontact structure on pseudo-Riemannian manifold. Recently, G\"und\"uzalp and \c Sahin  studied paracontact structures in  \cite{Gunduzlap:SlantsubmersionsfromLorentzian,gunduzlapShahin:ParacontactSemiRiemannianSubmersions,GUNDUZALP:paracontactparacomplexsemiRiemanniansubmersion}.

In this paper we studied anti-invariant semi-Riemannian submersion  from Lorentzian (para)almost contact manifolds.
In Sect. 3, we introduced anti-invariant semi-Riemannian submersion from Lorentzian (para)almost contact manifolds and presented three  examples. Also we find   necessary and sufficient conditions for the characteristic vector field to be vertical or horizontal. In sect. 4, we studied anti-invariant semi-Riemannian submersion  from Lorentzian (para)Sasakian manifolds onto a Riemannian manifold such that the characteristic vector field is vertical and  investigated the geometry of leaves of the distributions. In sect. 5,  we studied anti-invariant semi-Riemannian submersion from Lorentzian (para)Sasakian manifolds onto a Lorentzian manifold such that the characteristic vector field is horizontal and we obtained decomposition theorems for it.


\section{Preliminaries}
In this section, we recall some necessary details background on Lorentzian almost contact manifold, Lorentzian almost para contact manifold, semi-Riemannian submersion and harmonic maps.
\subsection{Lorentzian almost contact manifold}
Let $(M ,g)$ be a $(2n+1)$-dimensional Lorentzian  manifold with a tensor field $\phi$ of type  $(1, 1)$, a vector field $\xi$ and  a 1-form  $\eta $ 
which satisfy
\begin{align}
&\phi ^{2} X=\varepsilon X+\eta (X) \xi ,\label{3.1}\\
&g(\phi X,\phi Y)=g(X,Y) + \eta(X) \eta(Y),\label{3.2}\\
&\eta (X)= \varepsilon g(X,\xi),\label{3.3}\\
&\eta (\xi )=- \varepsilon\label{3.1.4} ,
\end{align}
for any vector fields $X, Y$ tangent to   $ M $, it is called Lorentzian almost contact manifold or Lorentzian almost para contact manifold for $\varepsilon= -1$ or $\varepsilon=1$, respectively\cite{Alegre:Slantsubmanifolds}. In this case from \eqref{3.1} and \eqref{3.1.4} imply that $\phi\xi=0, \eta o\phi=0$, and $rank \phi=2n$. However, for any vector fields $X,Y$ in $\Gamma(TM)$,
\begin{align}\label{5}
g(\phi X, Y)=\varepsilon g(X, \phi Y).
\end{align}
Let $ \Phi $ be the 2-form in $ M $  given by $\Phi (X,Y) = g(X, \phi Y ).$ Then, $M$ is called Lorentzian metric contact manifold if $d\eta (X,Y)=\Phi (X,Y).$
So, $M$ is called almost normal contact Lorentzian manifold if satisfying $[\phi , \phi]+2d\eta \otimes \xi =0.$ If $ \xi $ is a Killing tensor vector field, then the (para)contact structure is called K-(para)contact. In such a case, we have
\begin{align}
\nabla _{X} \xi= \varepsilon \phi X,\label{3.6}
\end{align}
where $ \nabla $ denotes the Levi-Civita connection of $g$.
A Lorentzian almost contact manifold or Lorentzian almost para contact manifold $M$ is called  Lorentzian Sasakian (LS) or Lorentzian para Sasakian (LPS) if
\begin{align}
(\nabla _{X} \phi) Y=g(\phi X,\phi Y) \xi+ \eta(Y) \phi ^{2} X.\label{3.7}
\end{align}
Now we will introduce a well known Sasakian manifold example on $\R ^{2n+1}.$
\begin{exam}[\cite{blair.Ri:book}] \label{exam:1}
Let $\R^{2n+1}=\{(x^1,\ldots,x^n,y^1,\ldots,y^n,z)|x^i,y^i,z\in\R, i=1,\ldots,n\}$. Consider $\R^{2n+1}$ with  the following structure:
\begin{align}
&\phi_{\epsilon} \left(\sum_{i=1}^{n}( X_{i}\frac{\partial}{\partial x^{i}}+Y_{i}\frac{\partial}{\partial y^{i}})+Z\frac{\partial}{\partial z} \right)=-\epsilon\sum_{i=1}^{n} Y_{i}\frac{\partial}{\partial x^{i}}-\sum_{i=1}^{n} X_{i}\frac{\partial}{\partial y^{i}}+\sum_{i=1}^{n} Y_{i}y_{i}\frac{\partial}{\partial z}\label{fi},\\
&g=-\eta \otimes \eta +\frac{1}{4}\sum_{i=1}^{n}( dx^{i}\otimes dx^{i}+ dy^{i}\otimes dy^{i})\label{eq:metr},\\
&\eta_\epsilon=-\frac{\epsilon}{2}\left(dz-\sum_{i=1}^{n} y^{i}dx^{i} \right)\label{eq:eta},\\
&\xi=2\frac{\partial}{\partial z}.
\end{align}
Then, $ (\R^{2n+1},\phi_{\epsilon},\xi,\eta_{\epsilon},g), $ is a Lorentzian Sasakian manifold if $\epsilon=-1$ and  Lorentzian para Sasakian manifold if $\epsilon=1$.
The vector fields $E_{i} =  2\frac{\partial}{\partial y^{i}} ,E_{n+i}=2(\frac{\partial}{\partial x^{i}}+ y_{i} \frac{\partial}{\partial z})$ and  $\xi$ form a $\phi$-basis for the contact metric structure.
\end{exam}
\subsection{Semi-Riemannian submersion}

Let $(M,g_M)$ and $(N,g_N)$ be semi-Riemannian manifolds. A semi-Riemannian submersion   $F:M\to N$ is a submersion of
semi-Riemannian manifolds such that:
\begin{enumerate}
  \item The fibers $F^{-1}(q), q\in N,$ are semi-Riemannian submanifolds  of $M.$
  \item $F_*$ preserves scalar  products of vectors normal to fibers.
\end{enumerate}

For each $q\in N, F^{-1}(q)$ is a submanifold of $M$ of dimension  $\dim M-\dim N.$ The submanifolds $F^{-1}(q), q\in N$ are called fibers, and a vector field on $M$ is vertical if it is always tangent to fibers, horizontal if always orthogonal to fibers.
A vector field $X$ on $M$ is called basic if $X$ is horizontal and
$F$-related to a vector field $X_*$ on $N$. Every vector field $X_*$ on $N$ has a unique horizontal lift $X$ to $M$, and $X$ is basic. For a semi-Riemannian submersion $F:M\to N$, let $\h$ and $\V$ denote the projections of the tangent spaces of $M$ onto the subspaces of horizontal and vertical vectors, respectively. In the other words, $\h$ and $\V$ are the projection morphisms on the distributions $(\ker F_*)^\bot$  and $\ker F_*$, respectively \cite{oneil:book}.
\begin{lemma}[\cite{Oneill:Thefundamentalequationsofsubmersion}]
Let $F:M\to N$ be semi-Riemannian submersion between Semi-Riemannian manifolds and $X, Y$ be basic vector fields of $M$. Then
\begin{enumerate}
  \item[a)] $g_M(X,Y)=g_N(X_*,Y_*)oF,$
  \item[b)] the horizontal part $\h[X,Y]$ of $[X,Y]$ is a basic vector field and corresponds to $[X_*,Y_*],$ i.e., $F_*(\h[X,Y])=[X_*,Y_*].$
  \item[c)] $[V,X]$ is vertical for any vector field $V$ of $\ker F_*$.
  \item[d)] $\h(\co _X^M Y)$ is the basic vector field corresponding to $\co_{X_*}^N Y_*.$
\end{enumerate}
\end{lemma}
The fundamental tensors of a submersion were defined by O'Neill. They are $(1,2)$-tensors on $M$, given by the formula:
\begin{align}
\T(E,F)&=\T_EF=\h \co_{\V E}\V F+\V\co_{\V E}\h F,\label{2.1}\\
\A(E,F)&=\A_EF=\V \co_{\h E}\h F+\h\co_{\h E}\V F,\label{2.2}
\end{align}
for any vector field $E$ and $F$ on $M$, where $\co$ denotes the Levi-Civita connection of $(M,g_M).$ It is easy to see that a Riemannian
submersion $F : M \to N$ has totally geodesic fibers if and only if $T$ vanishes identically. For any $E\in \Gamma(TM)$, $\T_E$ and $\A_E$ are skew-symmetric operators on $(\Gamma(TM), g)$ reversing the horizontal and the vertical distributions. In the other words,
\begin{align}
g(\T_DE,G)=-g(E,\T_DG),\label{2.10}\\
g(\A_DE,G)=-g(E,\A_DG),\label{2.11}
\end{align}
for any $D, E, G\in\Gamma(TM).$
 It is also easy
to see that $\T$ is vertical, $\T_E = \T_{\V E}$ and $\A$ is horizontal, $\A = \A_{\h E}$. For any $U, V$ vertical and $X, Y$ horizontal vector fields $\T, \A$ satisfy:
\begin{align}
\T_U V&=\T_V U,\label{2.3}\\
\A_XY&=-\A_YX=\frac{1}{2}\V[X,Y].\label{2.4}
\end{align}
Moreover, from \eqref{2.1} and \eqref{2.2} we have
\begin{align}
\co_VW&=\T_VW+\hat{\co}_VW, \label{2.5}\\
\co_VX&=\h{\co}_VX+\T_VX, \label{2.6}\\
\co_XV&=\A_XV+\V{\co}_XV, \label{2.7}\\
\co_XY&=\h{\co}_XY+\A_XY, \label{2.8}
\end{align}
for $X,Y\in\Gamma((\ker F_*)^\bot)$ and $V, W\in \Gamma(\ker F_*),$ where $\hat{\co}_VW=\V\co_VW.$
\subsection{Foliations on manifold and decomposition theorem}\label{decomposition}
A foliation $\di$ on a manifold $M$ is an integrable distribution.  A foliation $\di$ on a semi-Riemannian manifold $M$ is called totally umbilical, if every leaf of $\di$ is a totally umbilical semi-Riemannian submanifold of $M$.
If, in addition, the mean curvature vector
of every leaf is parallel in the normal bundle, then $\di$ is called a sphenic
foliation, because in this case each leaf of $\di$ is an extrinsic sphere of $M$.
If every leaf of $\di$ is a totally geodesic submanifold of $\di$, then $\di$ is called
a totally geodesic foliation\cite{chen:2011pseudo}.The following  results were proved in\cite{ponge:twisted.product}.

Let $(M, g)$ be a simply-connected semi-Riemannian manifold
which admits two complementary foliations $\di_1$ and $\di_2$ whose leaves
intersect perpendicularly. \\
1. If $\di_1$ is totally geodesic and $\di_2$ is totally umbilical,
then $(M, g)$ is isometric to a twisted product $M_1\times_fM_2$.\\
2. If $\di_1$ is totally geodesic and $\di_2$ is spherical, then $(M, g)$ is isometric to a warped product $M_1\times_fM_2$.\\
3. If $\di_1$ and $\di_2$ are totally geodesic, then $(M, g)$ is isometric to a direct  product $M_1\times M_2$, where $M_1$ and $M_2$ are integral manifolds of distributions $\di_1$ and $\di_2$.
\subsection{Harmonic maps}
We now recall the notion of harmonic maps between semi-Riemannian manifolds.
Let $(M, g_M)$ and $(N, g_N)$ be semi-Riemannian manifolds and suppose that $\varphi: M \to N$ is a smooth mapping between them. Then the differential $\varphi_*$ of $\varphi$ can be viewed a section
of the bundle $Hom(TM,\varphi^{-1}TN)\to M$, where $\varphi^{-1}TN$ is the pullback bundle
which has fibers $\varphi^{-1}(TN_p)=T_{\varphi(p)}N, p\in M$. $Hom(TM,\varphi^{-1}TN)$ has a connection $\nabla$ induced from the Levi-Civita connection
$\nabla^M$ and the pullback connection. Then the second fundamental form of $\varphi$ is
given by
\begin{align}\label{2.12}
(\nabla\varphi_*)(X,Y)=\nabla_X^\varphi \varphi_*(Y)-\varphi_*(\nabla_X^MY)
\end{align}
for $X, Y\in\Gamma(TM)$, where $\nabla^\varphi$ is the pullback connection. It is known that the second
fundamental form is symmetric. For a Semi-Riemannian submersion $F$, one can easily obtain
\begin{align}\label{2.13}
(\nabla F_*)(X,Y)=0
\end{align}
for $X, Y\in\Gamma((\ker F_*)^\bot)$. A smooth map $\varphi: M \to N$ is said to be harmonic if $\tr(\nabla\varphi_*)=0.$ On the other hand, the tension field of $\varphi$ is the section
$\tau(\varphi)$ of $\Gamma(\varphi^{-1}TN)$ defined by
\begin{align}\label{2.14}
\tau(\varphi)=\diver \varphi_*=\sum_{i=1}^m\epsilon_i(\nabla\varphi_*)(e_i,e_i),
\end{align}
where $\{e_1,\ldots, e_m\}$ is the orthonormal frame on $M$ and $\epsilon_i=g_M(e_i,e_i)$. Then it follows that $\varphi$ is harmonic
if and only if $\tau(\varphi)=0$, for details, see \cite{Fuglede:Harmonicmorphismsbetween}.

\section{Anti-invariant semi-Riemannian submersions}
In this section, we study a semi-Riemannian submersion from a Lorentzian almost (para)contact manifold $M(\phi , \xi ,\eta , g_{M}) $ to a semi-Riemannian manifold $(N, g _{N})$ and give  necessary and sufficient conditions for the characteristic vector field to be vertical or horizontal.
\begin{defi} \label{defi:anti}
Let $M(\phi , \xi ,\eta , g_{M}) $
be a  Lorentzian almost (para)contact manifold and
$ (N, g _{N}) $
be a semi-Riemannian manifold. A semi-Riemannian submersion
$ F: M(\phi , \xi ,\eta , g_{M}) \rightarrow (N, g _{N})$
is said to be anti-invariant if
$ \ker F_{\ast} $
is anti-invariant with respect to
$ \phi $,
$\phi (\ker F_{\ast}) \subseteq (\ker F_{\ast}) ^{\perp}. $
\end{defi}

 We denote the complementary orthogonal distribution to $ \phi (\ker F_{*}) $ in $ (\ker F_{*}) ^{\perp} $by $ \mu $. Then we have
\begin{align}\label{25}
(\ker F_{*}) ^{\perp}=\phi (\ker F_{*}) \oplus \mu.
\end{align}
\subsection{Examples}
We now give some examples of anti-invariant semi-Riemannian submersion.
\begin{exam}\label{exam:mesalanti1}
Let $N$ be $\R ^{5}=\{(y_1,y_2,y_3,y_4,z)|y_1,y_2,y_3,z\in\mathbb{R}\}$
and $\R^{7}$ be a Lorentzian Sasakian manifold as in Example \ref{exam:1}. The semi-Riemannian metric tensor field $g_{N}$ is given by
\[
g_{N}=\frac{1}{4}
\begin{pmatrix}
\frac{1}{2}-y_1^2 & -y_1y_2 & -y_1y_3 & 0 & y_1\\[2mm]
-y_1y_2 & \frac{1}{2}-y_2^2 & -y_2y_3 & 0 & y_2\\[2mm]
-y_1y_3 & -y_2y_3 & \frac{1}{2}-y_3^2 & 0 & y_3\\[2mm]
0 & 0 & 0 & \frac{1}{2} & 0\\
y_1 & y_2 & y_3 & 0 & -1
\end{pmatrix}
\]
on $N$.
Let $F: \R ^{7} \rightarrow N$ be a map defined by
$$ F(x_{1}, x_{2}, x_{3}, y_{1}, y_{2}, y_{3} , z)=(x_{1}+y_{1}, x_{2}+y_{2}, x_{3}+y_{3} , x_{3}-y_{3}, \frac{y_{1} ^{2}}{2}+\frac{y_{2} ^{2}}{2}+\frac{y_{3} ^{2}}{2}+z ). $$
After some calculations we have
$\ker F_{*}= \Span \lbrace V_{1}=E_{1}-E_{4}, V_{2}=E_{2}-E_{5} \rbrace$
and
$$ \ker F_{*} ^{\perp}= \Span \lbrace H _{1}=E _{1}+E _{4}, H _{2}=E_{2}+E_{5} , H_{3}=E_{3} , H_{4}=E_{6}, H_{5}=E_{7} \rbrace .$$
It is easy to see that $F$ is a semi-Riemannian submersion and $\phi_{-1}(V_1)=H_1, \phi_{-1}(V_2)=H_2$ imply that $\phi_{-1}(\ker F_*)\subset(\ker F_*)^\bot=\phi_{-1}(\ker F_*)\oplus\Span\{H_3,H_4,H_5\}.$ Thus $F$ is an anti-invariant semi-Riemannian submersion such that $\xi$ is horizontal and $\mu=\Span\{H_3,H_4,H_5\}$. Moreover, $\phi_{-1}(\ker F_*)$ is Riemannian Distribution.

It is clear that $ F:(\R ^{7},\phi_1,\eta_1,\xi,g)\to N $ is anti-invariant semi-Riemannian submersion from  Lorentzian para Sasakian  manifold to semi-Riemannian manifold.
\end{exam}
\begin{exam} \label{ex:mesalanti2}
$\R^{5}$ has a Lorentzian Sasakian structure as in Example \ref{exam:1}. The Riemannian metric
tensor field $g_{\R ^{2}}$ is defined by
 $
g_{\R ^{2}}=\frac{1}{8}(du\otimes du+dv\otimes dv) $ on $\R^2=\{(u,v)|u,v\in\R\}$.
 Let $F:\R^{5}\to\R^{2}$ be a map defined by $F(x_{1},x_{2},y_{1},y_{2},z)=(x_{1} + y_{1} , x_{2} + y_{2}).$ Then, by direct calculations
$\ker F_{*}= \Span \lbrace V_{1}=E_{1}-E_{3}, V_{2}=E_{2}-E_{4}, V_{3}=E_{5}=\xi \rbrace$
and
$ (\ker F_{*}) ^{\perp}=\Span \lbrace H _{1}=E _{1}+E _{3}, H _{2}=E_{2}+E_{4} \rbrace .$
Then it is easy to see that $F$ is a semi-Riemannian submersion. However, $\phi_{-1}(V_1)=H_1, \phi_{-1}(V_2)=H_2$. That is, $F$ is an anti-invariant semi-Riemannian submersion and $\phi(\ker F_*)=(\ker F_{*}) ^{\perp}.$
So, $F$ from para Sasakian Lorentzian manifold $ (\R ^{5},\phi_1,\eta_1,\xi,g)$ to Riemannian manifold $(\R^2,g_{\R})$ is anti-invariant.
\end{exam}

\begin{exam}\label{exam:mesalanti1}
Let $N$ be $\R ^{3}=\{(y_1,y_2,z)|y_1,y_2,z\in\mathbb{R}\}$
and $\R ^{5}$ be a Lorentzian Sasakian manifold as in Example \ref{exam:1}.
The Lorentzian metric tensor field $g_{N}$ is given by
\[
g_{N}=\frac{1}{4}
\begin{bmatrix}
\frac{1}{2}-y_1^2 & -y_{1} y_{2} & y_{1} \\[2mm]
-y_{1} y_{2} & \frac{1}{2}-y_2^2 & y_2 \\[2mm]
y_1 & y_2 & -1
\end{bmatrix}
\]
on $N$.
Let $F: \R ^{5} \rightarrow N$ be a map defined by
$$ F(x_{1}, x_{2}, y_{1}, y_{2}, z)=(x_{1} + y_{1}, x_{2} +y_{2} , \frac{y_{1} ^{2}}{2} +  \frac{y_{2} ^{2}}{2}+z ). $$
After some calculations we have
$\ker F_{*}=\Span \lbrace V_{1}=E_{3}-E_{1}, V_{2}=E_{4}-E_{2} \rbrace$
and
$ (\ker F_{*}) ^{\perp}=\Span \lbrace H _{1}=E _{1}+E _{3}, H _{2}=E_{2}+E_{4} , H_{3}=E_{5} \rbrace .$
Then it is easy to see that $F$ is an anti-invariant semi-Riemannian submersion and $(\ker F_*)^\bot=\phi_{-1}(\ker F_*)\oplus\Span\{\xi\}.$
\end{exam}

In the following results, we find necessary and sufficient conditions for the characteristic vector field to be vertical or horizontal.
\begin{theorem}\label{pro:1}
Let $M(\phi, \xi, \eta , g_{M})$ be a Lorentzian almost (para)contact manifold of dimension $2m+1$ and $(N, g_{N})$ be a semi-Riemannian manifold of dimension $n$ and  $F:M(\phi,\xi,\eta, g_{M}) \to (N, g_{N})$ be a semi-Riemannian submersion.
 \begin{enumerate}
   \item  The characteristic vector field $\xi$ is vertical if and only if $N$ is a Riemannian manifold.
   \item  The characteristic vector field $\xi$ is horizontal  if and only if $N$ is a Lorentzian manifold.
 \end{enumerate}
\end{theorem}
\begin{proof}
Let $F$ be a semi-Riemannian submersion. Then $F_*$ is an isometry from  $(\ker F_*)^\bot_p$ to $T_{F(p)}N$ for every point $p$ of $M$. So, they have the same dimension and index. $\xi$ is (horizontal)vertical if and only if (horizontal)vertical distribution is Lorentzian distribution and  (vertical)horizontal distribution is Riemannian distribution.
\end{proof}

\begin{theorem} \label{theorem:01}
Let $M(\phi, \xi, \eta , g_{M})$ be a Lorentzian almost (para)contact manifold of dimension $2m+1$ and $(N, g_{N})$ be a semi-Riemannian manifold of dimension $n$. Let $F:M(\phi,\xi,\eta, g_{M}) \to (N, g_{N})$ be an anti-invariant semi-Riemannian submersion.
\begin{enumerate}
\item[(a)] If the characteristic vector field $\xi$ is vertical then $ m\leqslant n\leqslant 2m$.
\item[(b)] If $ m=n $ then the characteristic vector field $\xi$ is vertical.
\item[(c)] If the characteristic vector field $\xi$ is horizontal then $ m+1\leqslant n$.
\end{enumerate}
\end{theorem}
\begin{proof}
\textbf{ Proof of (a).} Assume that the characteristic vector field $\xi$ is vertical. We have  $ 0\leqslant \dim\phi (\ker F_{*})=2m-n\leqslant n $, then $ m\leqslant n\leqslant 2m$.\\
\textbf{ Proof of (b).} Assume that $m=n$ and $k=\dim\{X\in\ker F_*|\phi(X)=0\}$. If $\xi$ is not vertical, then $k=0$. Therefore, $\dim\phi(\ker F_*)=n+1\leqslant n$, it is a contradiction.\\
\textbf{ Proof of (c).} If the characteristic vector field $\xi$ is horizontal, then $\dim\phi(\ker F_*)=2m+1-n\leqslant n$. Therefore, $1\leqslant2(n-m)$, we have $1\leqslant n-m.$
\end{proof}
\begin{theorem}
Let $F$ be a semi-Riemannian submersion from a $K$-(para)contact manifold $ M(\phi, \xi, \eta, g_{M})$ of dimension $2m+1$ onto a semi-Riemannian manifold $(N, g_{N})$ of dimension $n$. If $ \xi $ is horizontal, then F is an  anti-invariant submersion and  $ m+1\leqslant n$.
\end{theorem}
\begin{proof}
From \eqref{3.6}, \eqref{2.10} and \eqref{2.3} we have
\begin{align*}
g_M(\phi U,V)=g_M(\varepsilon \nabla_{U}\xi ,V)=\varepsilon g_M(\mathcal{T} _{U} \xi , V)=-\varepsilon g_M(\xi ,\mathcal{T} _{U}V)
\end{align*}
for any $U, V \in \Gamma(\ker  F_{*}) $. Since $\phi$ is skew-symmetric and $\mathcal{T}$ is symmetric, that is, \eqref{2.6}, we have $g_M(\phi U,V)=0$. Thus $F$ is an anti-invariant submersion. From part (c) of Theorem \ref{theorem:01} we have  $ m+1\leqslant n$.
\end{proof}
\begin{cor}\label{cor:1}
Let $M(\phi, \xi, \eta , g_{M})$ be a Lorentzian almost (para)contact manifold of dimension $2m+1$ and $(N, g_{N})$ is a semi-Riemannian manifold of dimension $n$ and $F:M(\phi,\xi,\eta, g_{M}) \to (N, g_{N})$ be an anti-invariant semi-Riemannian submersion. If $m = n$, then $ \phi (\ker F_{*}) = (\ker F_{*}) ^{\perp}$. Moreover, $N$ is a Riemannian manifold.
\end{cor}
\begin{pro} \label{theorem:1}
Let $M(\phi, \xi, \eta , g_{M})$ be a Lorentzian almost (para)contact manifold of dimension $2m+1$ and $(N, g_{N})$ is a semi-Riemannian manifold of dimension $n$ and $F:M(\phi,\xi,\eta, g_{M}) \to (N, g_{N})$ be an anti-invariant semi-Riemannian submersion such that $ \phi (\ker F_{*}) = (\ker F_{*}) ^{\perp}$. Then the characteristic vector field $\xi$ is vertical and $m = n$. Moreover, $N$ is a Riemannian manifold.
\end{pro}
\begin{proof}
If $\xi$ is not vertical, then  $\dim\phi(\ker F_*)=2m+1-n=n$. Therefore,  $2(n-m)=1$, it is a contradiction. So $\xi\in\ker F_*$. That is, $\xi$ is vertical. Now since $\xi $ is vertical we have $\dim\phi(\ker F_*)=2m-n=n$. Thus $m=n$ and by Theorem \ref{pro:1},  $N$ is a Riemannian manifolds.
\end{proof}
\begin{pro} \label{theorem:2}
Let $M(\phi, \xi, \eta , g_{M})$ be a Lorentzian almost (para)contact manifold of dimension $2m+1$ and $(N, g_{N})$ be a semi-Riemannian manifold of dimension $n$ and $F:M(\phi,\xi,\eta, g_{M}) \to (N, g_{N})$ be an anti-invariant semi-Riemannian submersion such that $\phi(\ker F_{*}) =\{0\}$. Then the characteristic vector field $\xi$ is vertical, $2m =n$ and $\ker F_*=\Span\{\xi\}$. Moreover, $N$ is a Riemannian manifolds.
\end{pro}
\begin{proof}
If $\xi$ is not vertical, then  $\dim\phi(\ker F_*)=2m+1-n=0$. Therefore,  $\dim\ker F_*=0$, it is contraction. So $\xi$ is vertical. In this case   $\dim\phi(\ker F_*)=2m-n=0$ and $\dim\ker F_*=1$, Thus $2m=n,  \ker F_*=\Span\{\xi\}$  and by Theorem \ref{pro:1},  $N$ is a Riemannian manifolds.
\end{proof}
\begin{pro} \label{theorem:2}
Let $M(\phi, \xi, \eta , g_{M})$ be a Lorentzian almost (para)contact manifold of dimension $2m+1$ and $(N, g_{N})$ be a semi-Riemannian manifold of dimension $n$ and $F:M(\phi,\xi,\eta, g_{M}) \to (N, g_{N})$ be an anti-invariant semi-Riemannian submersion. If $2m =n$, then $\xi$ is vertical, $\ker F_*=\Span\{\xi\},  \phi (\ker F_{*}) =\{0\}$ and  $N$ is a Riemannian manifolds or  $\xi$ is horizontal and $N$ is a Lorentzian manifolds
\end{pro}
\begin{proof}
If $\xi$ is not vertical, then  $\dim\phi(\ker F_*)=2m+1-n=0$. Therefore,  $\dim\ker F_*=0$, it is contraction. So $\xi$ is vertical. In this case   $\dim\phi(\ker F_*)=2m-n=0$ and $\dim\ker F_*=1$, Thus $2m=n,  \ker F_*=\Span\{\xi\}$  and by Theorem \ref{pro:1},  $N$ is a Riemannian manifolds.
\end{proof}
\begin{pro} \label{theorem:7}
Let $M(\phi, \xi, \eta , g_{M})$ be a Lorentzian almost (para)contact manifold of dimension $2m+1$ and $(N, g_{N})$ is a Lorentzian Riemannian manifold of dimension $n$. Let $F:M(\phi,\xi,\eta, g_{M}) \to (N, g_{N})$ be an anti-invariant semi-Riemannian submersion.  $(\ker F_*)^\perp=\phi(\ker F_*)\oplus\Span\{\xi\}$ if and only if $m+1 = n$.
\end{pro}
\begin{proof}
 Obviously, $\xi$ is horizontal, if $(\ker F_*)^\perp=\phi(\ker F_*)\oplus\Span\{\xi\}$ then  $\dim\phi(\ker F_*)=2m+1-n=n-1$, so  $m+1 = n$. Conversely, by using \eqref{25}, we have $2m+1-n+\dim\mu=n$. So $\dim\mu=1 $ then $\mu=\Span\{\xi\}$. 
\end{proof}
\begin{remark}
We note that Example \ref{exam:mesalanti1} satisfies Proposition \ref{theorem:7}.
\end{remark}
\section{Anti-invariant submersions admitting vertical structure vector field}
In this section, we will study anti-invariant submersions from a  Lorentzian (para) Sasakian manifold onto a
Riemannian manifold such that the characteristic vector field $\xi$ is vertical.
It is easy to see that $ \mu $ is an invariant distribution of $ (\ker F_{*})^\bot $,
under the endomorphism $ \phi $.
Thus, for $ X \in \Gamma ((\ker F_{*})^\bot)$ we write
\begin{align}
\phi X=BX+CX. \label{4.1}
\end{align}
where $BX \in \Gamma (\ker F_{*}), CX \in \Gamma (\mu)$. On the other hand, since $ F_{*} ((\ker F_{*}) ^{\perp})=TN $ and $F$ is a semi-Riemannian submersion, using \eqref{4.1} we derive $g_{N}(F_{*}\phi V, F_{*} CX) = 0$, for every $ X \in \Gamma ((\ker F_{*})^\perp ) , V \in \Gamma (\ker F_{*})$ which implies that
\begin{align}
TN=F_{*} (\phi (\ker F_{*}) ^{\perp}) \oplus F_{*} (\mu). \label{4.2}
\end{align}
\begin{theorem} \label{theorem:2}
Let $M(\phi, \xi, \eta , g_{M})$ be a Lorentzian almost (para)contact manifold of dimension $2m+1$ and $(N, g_{N})$ be a Riemannian manifold of dimension $n$. Let $F:M(\phi,\xi,\eta, g_{M}) \rightarrow (N, g_{N})$ be an anti-invariant semi-Riemannian submersion and $\xi$ is vertical vector field. Then the fibers are not totally umbilical.
\end{theorem}
\begin{proof}
From \eqref{2.5}  we have that, for $U\in\Gamma(\ker F_*)$:
$\nabla _{U} \xi =\T _{U} \xi + \V \nabla _{U} \xi .$
And from \eqref{3.6} we have
$\nabla _{U} \xi =\varepsilon \phi U .$
So, we will have:
\begin{align}
\varepsilon \phi U= \T _{U} \xi .
\end{align}
If the fibers are totally umbilical, then we have $\T_{U} V = g_{M} (U, V )H$ for any vertical vector fields U, V where H is the mean curvature vector field of any fibers. Since $\T_{\xi} \xi = 0$, we have $H = 0$, which shows that fibres are minimal. Hence the fibers
are totally geodesic, which is a contradiction to the fact that $\T_{U} \xi =\varepsilon \phi U \neq 0$.
\end{proof}
\begin{lemma}
Let $F$ be a anti-invariant semi-Riemannian submersion from a Lorentzian (para)Sasakian manifold $M(\phi, \xi, \eta , g_{M})$ onto a Riemannian manifold $(N, g_{N})$. Then we have
\begin{align}
 BCX=0,
 C^2X+\phi BX=\epsilon X,\label{30}\\
 \nabla _{X} Y=  g(X , \phi Y)\xi+\varepsilon\phi \nabla _{X} \phi Y,\label{31}
\end{align}
where $X,Y \in \Gamma ((\ker F_{*}) ^{\perp}) $.
\end{lemma}
\begin{proof}
First, By using \eqref{3.1} and \eqref{4.1} for $X\in\Gamma(\ker F_*)$ we obtain $\epsilon X=BCX+C^2X+\phi BX.$ This proves \eqref{30}. Next, \eqref{31} is obtained from \eqref{3.1}, \eqref{3.6} and \eqref{3.7}.
\end{proof}
\begin{lemma}
Let $F$ be an anti-invariant semi-Riemannian submersion from a Lorentzian (para)Sasakian manifold $M(\phi, \xi, \eta , g_{M})$ onto a Riemannian manifold $(N, g_{N})$. Then we have
\begin{align}
&CX=\varepsilon \A_{X} \xi,\label{4.6}\\
&g_{M} (\A_{X} \xi , \phi U)=0,\label{4.7}\\
&g_{M}(\nabla _{Y} \A_{X} \xi ,\phi U) =-g_{M}(\A_{X} \xi , \phi \A_{Y} U)-\varepsilon \eta(U) g_{M} (\A_{X} \xi ,Y),\label{4.8}\\
&g_{M}(X, \A_{Y} \xi )=\varepsilon g_{M}(Y, \A_{X} \xi),\label{4.9}
\end{align}
where $X,Y \in \Gamma ((\ker F_{*}) ^{\perp})$ and $U \in \Gamma (\ker F_{*})$.
\end{lemma}
\begin{proof}
By using from \eqref{2.7} and \eqref{3.6} for $X\in \Gamma ((\ker F_{*}) ^{\perp})$ and $V=\xi$, the equality \eqref{4.6}  is obvious.
Next, from \eqref{3.2}, \eqref{4.1} and \eqref{4.6}, the equality \eqref{4.7}  is obtained.
Now from \eqref{4.7} for  $X,Y \in \Gamma ((\ker F_{*}) ^{\perp})$, we get
$
g_{M} (\nabla _{Y} \A_{X} \xi , \phi U)+g_{M} (\A_{X} \xi , \nabla _{Y}\phi U)=0$ and $
g_{M} (\A_{X} \xi , \nabla _{Y}\phi U)=g_{M} \big(\A_{X}\xi , (\nabla _{Y} \phi )U \big)+g_{M}\big( \A_{X}\xi ,\phi (\nabla _{Y} U) \big)$.
By using \eqref{3.7} and \eqref{2.7} we obtain
\begin{align*}\begin{split}
g_{M} (\A_{X} \xi , \nabla _{Y}\phi U)=& \varepsilon g_{M}\big( \A_{X}\xi ,\eta (U) Y \big)
+ g_{M} \big(\A_{X}\xi , \phi \A_{Y} U \big)\\&+g_{M}\big( \A_{X}\xi ,\phi (\V \nabla _{Y}U)\big).
\end{split}
\end{align*}
Finally, by using \eqref{4.6}, \eqref{4.8} is obtained.
From \eqref{5}, \eqref{3.6} and \eqref{4.6}, we have \eqref{4.9}.
\end{proof}
\begin{theorem}
Let $F$ be an anti-invariant semi-Riemannian submersion from a Lorentzian (para) Sasakian manifold $M(\phi, \xi, \eta , g_{M})$ onto a Riemannian manifold $(N, g_{N})$, then the following assertions are equivalent to each other:
\begin{enumerate}
\item[$ (i) $]
$ (\ker F_{*})^{\perp} $ is integrable.
\item[$ (ii) $]
$g_{N} \big((\nabla F_{*})(Y,BX), F_{*} \phi V \big)=g_{N} \big((\nabla F_{*})(X,BY), F_{*} \phi V \big)+\varepsilon  g_{M} (A_{X} \xi , \phi A_{Y} V)-\varepsilon g_{M} (A_{Y} \xi , \phi A_{X} V).$
\item[$ (iii) $]
$ g_{M}(A_{X} BY- A_{Y} BX ,\phi V) = \varepsilon g_{M}(A_{X} \xi , \phi A_{Y} U) - \varepsilon g_{M} (A_{Y} \xi ,\phi A_{X} V).$
\end{enumerate}
\end{theorem}
\begin{proof}
$ (i)\Longleftrightarrow (ii).$ Assume that $U,V\in\Gamma(\ker F_*)$ and $X,Y\in\Gamma((\ker F_*)^\perp).$ From \eqref{31} and \eqref{5}, we obtain.
\begin{align*}
g_{M} ([X,Y],V)=&g_{M} (\nabla _{X}Y , V)-g_{M} (\nabla _{Y} X,V)\\
=&g_{M} (\varepsilon\phi \nabla _{X}\phi Y ,V)+g_{M} \big(g_M(Y,\phi X)\xi ,V \big)\\
&-g_{M}\big(\varepsilon\phi \nabla _{Y} \phi X,V )- g_{M} \big( g_M(X, \phi Y)\xi ,V \big)\\
=&g_{M} ( \nabla _{X}\phi Y ,\phi V)-g_{M}\big( \nabla _{Y} \phi X,\phi V)+(1-\varepsilon)\varepsilon g_{M} (\phi X,Y)\eta(V).
\end{align*}
Now from \eqref{4.1} , \eqref{4.6} and since $ F $ is an anti-invariant submersion, we have
\begin{align*}
g_{M} ([X,Y],V)=&g_{N} (F_{*} \nabla _{X} BY , F_{*}\phi V)+\varepsilon g_{M} (\nabla _{X} \A_{Y} \xi ,\phi V)- g_{N} (F_{*} \nabla _{Y}B X ,F_{*}\phi V)\\
&-\varepsilon g_{M} (\nabla _{Y} \A_{X} \xi ,\phi V)+(1-\varepsilon)g_{M} (\A_{X}\xi,Y)\eta(V).
\end{align*}
On the other hand, According to \eqref{2.12}, \eqref{4.8} and \eqref{4.9} we get
\begin{align}\begin{split}
g_{M} ([X,Y],V)=&-g_{N} (\nabla F_{*}({X}, BY), F_{*}\phi V)+\varepsilon g_{M} (\A_{Y} \xi ,\phi \A_{Y}V)\\
&+g_{N} (\nabla F_{*}({Y},BX) ,F_{*} \phi V)-\varepsilon g_{M} ( \A_{X} \xi , \phi \A _{Y} V)\end{split}
\end{align}
$ (ii)\Longleftrightarrow (iii)$.
By using from \eqref{2.7} , \eqref{2.12} and assume we have
\begin{align*}
g_{N} ( F_{*} \nabla _{Y} BX - \nabla _{X} BY, F_{*} \phi V)=g_{M} ( \A_{Y} BX ,\phi V)- g_{M} (\A_{X} BY, \phi V)
\end{align*}
Thus according to part $ (ii) $, we have
\begin{align}
 g_{M} ( \A _{Y} BX - \A _{X} BY, \phi V)=-\varepsilon g_{M} ( \A _{X} \xi , \phi \A _{Y} V)+\varepsilon g_{M} ( \A _{Y} \xi , \phi \A _{X} V)
\end{align}
\end{proof}
\begin{remark}
If $ \phi (\ker F_{*}) = (\ker F_{*}) ^{\perp}$ then we get $\varepsilon \A_X\xi=CX=0$ and $BX=\phi X$.
\end{remark}
Hence we have the following Corollary.
\begin{corollary}
Let $F:M(\phi,\xi,\eta, g_{M}) \rightarrow (N, g_{N})$ be an anti-invariant semi-Riemannian submersion such that $ \phi (\ker F_{*}) = (\ker F_{*}) ^{\perp}$, where $M(\phi, \xi, \eta , g_{M})$ is a Lorentzian (para) Sasakian manifold and $ (N, g_{N}) $ is a Riemannian manifold. Then for every $X,Y \in \Gamma (\ker F_{*})^{\perp} $, the following assertions are equivalent to each
other;
\begin{enumerate}
\item[$ (i) $]
$ (\ker F_{*})^{\perp} $ is integrable.
\item[$ (ii) $]
$(\nabla F_{*})(Y,\phi X)=(\nabla F_{*})(X, \phi Y)$.
\item[$ (iii) $]
$ A_{X} \phi Y= A_{Y} \phi X $.
\end{enumerate}
\end{corollary}
\begin{theorem}
Let $F:M(\phi,\xi,\eta, g_{M}) \rightarrow (N, g_{N})$ be an anti-invariant semi-Riemannian submersion, where $M(\phi, \xi, \eta , g_{M})$ is a Lorentzian (para) Sasakian manifold and $ (N, g_{N}) $ is a Riemannian manifold. Then the following assertions are equivalent to each other;
\begin{enumerate}
\item[$ (i) $]
$ (\ker F_{*})^{\perp} $ defines a totally geodesic foliation on $M$.
\item[$ (ii) $]
$ g_{M}(A_{X} BY, \phi V)= \varepsilon g_{M}(A_{Y} \xi , \phi A_{X} V) $.
\item[$ (iii) $]
$ g_{N}\big((\nabla F_{*})(X,\phi Y),F_{*} \phi V \big)=-\varepsilon g_{M}(A_{Y} \xi , \phi A_{X} V).$
\end{enumerate}
for every $X,Y \in \Gamma ((\ker F_{*}) ^{\perp})$ and $V \in \Gamma (\ker F_{*})$.
\end{theorem}
\begin{proof}$ (i)\Longleftrightarrow (ii).$ Assume that $V\in\Gamma(\ker F_*)$ and $X,Y\in\Gamma((\ker F_*)^\perp).$
By using \eqref{31} we have
\begin{align}
g_M(\nabla _{X} Y, V)=g_M(\nabla _{X} \phi Y, \phi V)+ \varepsilon \eta (V)g_M(X, \phi Y),\label{44}
\end{align}
and from \eqref{2.7}, \eqref{4.1} we have
\begin{align}
g_M(\nabla _{X} \phi Y, \phi V)=g_M(\A _{X} B Y, \phi V)+\varepsilon g_M(\nabla _{X} \A _{Y} \xi,\phi V),\label{45}
\end{align}
and too from \eqref{4.8} we have
\begin{align}
g_M(\nabla _{X} \phi Y,\phi V)=g_M(\A _{X} B Y, \phi V)-\varepsilon g_M( \A _{Y} \xi ,\phi \A _{X}V )-\eta(V) g_M(\A _{Y} \xi ,X) \label{46}
\end{align}
Now, from \eqref{4.1}, \eqref{4.6}, \eqref{44}, \eqref{45} and \eqref{46}, $ (\ker F_{*})^{\perp} $ is a totally geodesic foliation on $M$ if and only if
\begin{align}
g_{M}(\A _{X} B Y, \phi V)=\varepsilon g_{M}(\A _{Y} \xi, \phi \A _{X} V).
\end{align}
Finely,
By using from \eqref{2.12},\eqref{2.13}, \eqref{4.1}, \eqref{4.2} and \eqref{46} we have $ (ii)\Longleftrightarrow (iii).$
\end{proof}
\begin{corollary}
Let $F:M(\phi,\xi,\eta, g_{M}) \rightarrow (N, g_{N})$ be an anti-invariant semi-Riemannian submersion such that $ \phi (\ker F_{*}) = (\ker F_{*}) ^{\perp}$, where $M(\phi, \xi, \eta , g_{M})$ is a Lorentzian (para) Sasakian manifold and $ (N, g_{N}) $ is a Riemannian manifold. Then the  following assertions are equivalent to each other;
\begin{enumerate}
\item[$ (i) $]
$ (\ker F_{*})^{\perp} $ defines a totally geodesic foliation on M.
\item[$ (ii) $]
$ A_{X} \phi Y=0$.
\item[$ (iii) $]
$(\nabla F_{*})(X,\phi Y)=0$.
\end{enumerate}
for every $X,Y \in \Gamma ((\ker F_{*}) ^{\perp})$ and $V \in \Gamma (\ker F_{*})$.
\end{corollary}

\noindent
We note that a differentiable map $F$ between two semi-Riemannian manifolds is called totally geodesic if $ \nabla F_{*} =0 $. Using Theorem \ref{theorem:2} one can easily prove that the fibers are not totally geodesic. Hence we have the following Theorem.
\begin{theorem}
Let $F:M(\phi,\xi,\eta, g_{M}) \to  (N, g_{N})$ be an anti-invariant semi-Riemannian submersion such that $ \phi (\ker F_{*}) = (\ker F_{*}) ^{\perp}$, where $M(\phi, \xi, \eta , g_{M})$ is a Lorentzian (para) Sasakian manifold and $ (N, g_{N}) $ is a Riemannian manifold. Then $F$ is not totally geodesic map.
\end{theorem}
\noindent
Finally, we give a necessary and sufficient condition for an anti-invariant Riemannian
submersion to be harmonic.
\begin{theorem}
Let $F:M(\phi,\xi,\eta, g_{M}) \to (N, g_{N})$ be an anti-invariant semi-Riemannian submersion such that
 $m=n$, where $M(\phi, \xi, \eta , g_{M})$ is a Lorentzian (para) Sasakian manifold of dimension $2m+1$ and $ (N, g_{N}) $ is a Riemannian manifold of dimension $n$. Then $F$ is harmonic if and only if $\tr \phi (\T _{V}) =-n \eta (V )$, where  $V \in \Gamma (\ker F_{*})$.
\end{theorem}
\begin{proof}
We know that $ F $ is harmonic if and only if $ F $ has minimal fibres\cite{Eells}. Thus $F$ is harmonic if and only if
$\sum _{i=1} ^{k} \T _{e_{i}} e_{i}=0$, where $\{e_1,\ldots,e_{k-1},e_k=\xi\}$ is the orthonormal basis for $\ker F_*$ and  $k = 2m+1-n=n+1$ is dimension of $\ker F_{*} $.\\
On the other hand, from \eqref{2.5}, \eqref{2.6} and \eqref{3.7} we get
\begin{align}
g_M(\T _{V} \phi W,U)=\varepsilon g_M(\phi V, \phi W)\eta(U)+\eta(W)g_M(\phi^2V,U)+\varepsilon g_M(\T _{V} W,\phi U).\label{4.14}
\end{align}
By using \eqref{4.14} and \eqref{2.10} we get
\begin{align}
 -\varepsilon\sum _{i=1} ^{k} g_{M} ( e_{i} , \phi \T _{e_{i}}U)=\varepsilon \big((k-1)\eta(U)+g_{M} (\sum _{i=1} ^{k}\T _{e_{i}} e_{i} , \phi U)\big).
\end{align}
Since $ F $ is a Harmonic maping, $ \sum _{i=1} ^{k} (\T e_{i} e_{i} ,\phi U)=0 $. Then we have
 \begin{align}
 \tr \phi (\T _{U})=\sum _{i=1} ^{k} g_{M} (e_{i}, \phi\T _{e_{i}}U)=-n\eta(U).
\end{align}
\end{proof}
\section{Anti-invariant submersions admitting horizontal structure vector field}
In this section, we will study anti-invariant submersions from a  Lorentzian (para) Sasakian manifold onto a
Lorentzian manifold such that the characteristic vector field $\xi$ is horizontal. From \eqref{25},
it is easy to see that $\phi(\mu)\subset\mu$ and $\xi\in\mu$.
Thus, for $ X \in \Gamma ((\ker F_{*})^\bot)$ we write
\begin{align}
\phi X=BX+CX. \label{44.1}
\end{align}
where $BX \in \Gamma (\ker F_{*}), CX \in \Gamma (\mu)$. On the other hand, since $ F_{*} ((\ker F_{*}) ^{\perp})=TN $ and $F$ is a semi-Riemannian submersion, using \eqref{44.1} we derive $g_{N}(F_{*}\phi V, F_{*} CX) = 0$, for every $ X \in \Gamma ((\ker F_{*})^\perp ) , V \in \Gamma (\ker F_{*})$ which implies that
\begin{align}
TN=F_{*} (\phi (\ker F_{*})) \oplus F_{*} (\mu). \label{44.2}
\end{align}

\begin{lemma}
Let $F$ be an anti-invariant semi-Riemannian submersion from a Lorentzian (para) Sasakian manifold $M(\phi, \xi, \eta , g_{M})$ onto a Lorentzian  manifold $(N, g_{N})$. Then we have
\begin{align}
&BX=\varepsilon A_{X} \xi,\label{5.6}\\
& \T _{U} \xi =0, \label{5.7}\\
& g_{M}(\nabla _{X} CY, \phi U)=- g_{M}(CY, \phi \A_{X} U),\label{5.8}
\end{align}
where $X,Y \in \Gamma ((\ker F_{*}) ^{\perp})$ and $U \in \Gamma (\ker F_{*})$.
\end{lemma}
\begin{proof}
Assume that $X,Y \in \Gamma ((\ker F_{*}) ^{\perp})$ and $U \in \Gamma (\ker F_{*})$.
By using from \eqref{2.8} and \eqref{3.6}, we have
\begin{align}
BX=\varepsilon \A_{X} \xi,
\end{align}
and also from \eqref{2.6} and \eqref{3.6} we get
\begin{align}
\T _{U} \xi =0.
\end{align}
From \eqref{3.7} and  \eqref{2.7}, we obtain  \eqref{5.8}.
\end{proof}
\begin{theorem}\label{th:3.4}
Let $F$ be an anti-invariant semi-Riemannian submersion from a Lorentzian (para) Sasakian manifold $M(\phi, \xi, \eta , g_{M})$ onto a Lorentzian manifold $(N, g_{N})$.  Then the following assertions are equivalent.
\begin{enumerate}
\item[$ (i) $]
$ (\ker F_{*})^{\perp} $ is integrable.
\item[$ (ii) $]
\begin{align*}
g_{N} \big((\nabla F_{*})(Y,BX), F_{*} \phi V \big)=g_{N} \big((\nabla F_{*})(X,BY), F_{*} \phi V \big)- g_{M} (CX, \phi \A_{Y} V)\\
+  g_{M} (CY, \phi \A_{X} V)+\varepsilon g_{M} (X,\phi V) \eta(Y) - \varepsilon g_{M} (Y,\phi V) \eta (X).
\end{align*}
\item[$ (iii) $]
\begin{align*}g_{M}(\A_{X} \A_{Y} \xi - \A_{Y} \A_{X} \xi ,\phi V) =&- g_{M} (CX, \phi \A_{Y} V)+  g_{M} (CY, \phi \A_{X} V)\\
&+\varepsilon g_{M} (X,\phi V) \eta(Y) - \varepsilon g_{M} (Y,\phi V) \eta (X).\end{align*}
\end{enumerate}
for all $X,Y \in \Gamma ((\ker F_{*}) ^{\perp})$ and $V \in \Gamma (\ker F_{*})$.
\end{theorem}
\begin{proof}
Assume that $X,Y\in \Gamma ((\ker F_{*}) ^{\perp})$ and $V \in \Gamma (\ker F_{*})$. From \eqref{3.2}, \eqref{3.7} and \eqref{5}, we obtain.
\begin{align*}
g_{M} ([X,Y],V)=&g_{M} (\nabla _{X}Y , V)-g_{M} (\nabla _{Y} X,V)\\
=&g_{M} ( \nabla _{X}\phi Y ,\phi V)-\varepsilon\eta(Y)g_M(X,\phi V)\\
&-g_{M} ( \nabla _{Y}\phi X ,\phi V)+\varepsilon\eta(X)g_M(Y,\phi V)\\
=&g_{M} ( \nabla _{X}B Y ,\phi V)+g_{M} ( \nabla _{X}C Y ,\phi V)-\varepsilon\eta(Y)g_M(X,\phi V)\\
&-g_{M} ( \nabla _{Y}B X ,\phi V)-g_{M} ( \nabla _{Y}C X ,\phi V)+\varepsilon\eta(X)g_M(Y,\phi V).
\end{align*}
Since $ F $ is an anti-invariant submersion, we have
\begin{align*}
g_{M} ([X,Y],V)=&g_{N} ( F_*\nabla _{X}B Y ,F_*\phi V)+g_{M} ( \nabla _{X}C Y ,\phi V)-\varepsilon\eta(Y)g_M(X,\phi V)\\
&-g_{N} ( F_*\nabla _{Y}B X ,F_*\phi V)-g_{M} ( \nabla _{Y}C X ,\phi V)+\varepsilon\eta(X)g_M(Y,\phi V).
\end{align*}
On the other hand, according to \eqref{2.12}, \eqref{5.8} and \eqref{4.9} we get
\begin{align}\begin{split}
g_{M} ([X,Y],V)=&-g_{N} (\nabla F_{*}({X}, BY), F_{*}\phi V)- g_{M} (C Y ,\phi\A_X V)-\varepsilon\eta(Y)g_M(X,\phi V)\\
&+g_{N} (\nabla F_{*}({Y},BX) ,F_{*} \phi V)+ g_{M} (C X ,\phi\A_Y V)+\varepsilon\eta(X)g_M(Y,\phi V)\end{split}
\end{align}
which proves $ (i)\Longleftrightarrow (ii).$
By using from \eqref{2.7} , \eqref{2.12} and assume we have
\begin{align*}
g_{N} ( F_{*} \nabla _{Y} BX - \nabla _{X} BY, F_{*} \phi V)=-(g_{M} ( \A_{Y} BX ,\phi V)- g_{M} (\A_{X} BY, \phi V))
\end{align*}
Thus according to part $ (ii) $, we have $ (ii)\Longleftrightarrow (iii).$
\end{proof}
\begin{cor}\label{cor:3.5}
Let $F$ be an anti-invariant semi-Riemannian submersion from a Lorentzian (para) Sasakian manifold $M(\phi, \xi, \eta , g_{M})$ onto a Lorentzian manifold $(N, g_{N})$ with $(\ker F_*)^\perp=\phi(\ker F_*)\oplus\Span\{\xi\}$. Then the following assertions are equivalent.
\begin{enumerate}
\item[$ (i) $]
$ (\ker F_{*})^{\perp} $ is integrable.
\item[$ (ii) $]
$(\nabla F_{*})(Y,BX)=(\nabla F_{*})(X,BY)+\varepsilon\eta(Y)F_*X - \varepsilon\eta(X) F_*Y.$
\item[$ (iii) $]
$\A_{X}\A_{Y}\xi -\A_{Y}\A_{X}\xi =\varepsilon\eta(Y)X-\varepsilon\eta (X)Y.$
\end{enumerate}
for all $X,Y \in \Gamma ((\ker F_{*}) ^{\perp})$ and $V \in \Gamma (\ker F_{*})$.
\end{cor}
\begin{theorem}\label{theorem:3.6}
Let F be an anti-invariant semi-Riemannian submersion from a (para) Lorentzian Sasakian manifold $(M,g_M,\phi ,\xi,\eta)$ onto a Lorentzian manifold $ (N,g_N) $. Then the following are equivalent.
\begin{enumerate}
\item[(i)]
$ (\ker F_*)^{\perp} $ defines a totally geodesic foliation on M.
\item[(ii)]
$g_M(\mathcal{A}_X BY,\phi V)=g_M (CY,\phi \mathcal{A}_X V)+\varepsilon \eta(Y)g(X,\phi V) $
\item[(iii)]
$ g_N\big( (\nabla F_*)(Y,\phi X) , F_*(\phi V)\big) =g_M (CY,\phi \mathcal{A}_X V)+\varepsilon \eta(Y)g(X,\phi V) $
\end{enumerate}
\end{theorem}
\begin{proof}
For $X,Y \in \Gamma \big((\ker F_*)^{\perp}\big)$ and $V \in \Gamma (\ker F_*)$, from \eqref{3.2}, \eqref{3.7}  and \eqref{5.8} we obtain
$$g_M(\nabla_{X}Y,V)=g_M(\mathcal{A}_X BY,\phi V)-g_M (CY,\phi \mathcal{A}_X V)-\varepsilon \eta(Y)g(X,\phi V),$$
which shows $ (i)\Longleftrightarrow (ii).$ From \eqref{2.7} and \eqref{2.12} we have $ (ii)\Longleftrightarrow (iii).$
\end{proof}
\begin{cor}\label{cor:3.7}
Let F be an anti-invariant semi-Riemannian submersion from a (para) Lorentzian Sasakian manifold $(M,g_M,\phi ,\xi,\eta)$ onto a Lorentzian manifold $ (N,g_N) $ with $(\ker F_*)^\perp=\phi(\ker F_*)\oplus\Span\{\xi\}$. Then the following are equivalent.
\begin{enumerate}
\item[(i)]
$ (\ker F_*)^{\perp} $ defines a totally geodesic foliation on M.
\item[(ii)]
$\mathcal{A}_X BY=\varepsilon \eta(Y)X $
\item[(iii)]
$(\nabla F_*)(Y,\phi X) =\varepsilon \eta(Y)F_*X $
\end{enumerate}
\end{cor}

\begin{theorem}\label{theorem:3.8}
Let F be an anti-invariant semi-Riemannian submersion from a (para) Lorentzian Sasakian manifold $(M,g_M,\phi ,\xi,\eta)$ onto a Lorentzian manifold $ (N,g_N) $. Then the following are equivalent.
\begin{enumerate}
\item[(a)]
$ \ker F_* $ defines a totally geodesic foliation on M.
\item[(b)]
$ g_N \big( (\nabla F_*)(V,\phi X),F_* \phi W \big) =0$
\item[(c)]
$ \T_{V}BX+\mathcal{A}_{CX}V \in \Gamma(\mu) $, for $X \in \Gamma \big((\ker F_*)^{\perp}\big)$ and $ V,W \in \Gamma(\ker F_*) $.
\end{enumerate}
\end{theorem}
\begin{proof}
For $X \in \Gamma \big((\ker F_*)^{\perp}\big)$ and $ V,W \in \Gamma(\ker F_*) $, $ g_M (W,\xi)=0 $ implies that from \eqref{3.7},
$ g_M(\nabla_{V}W,\xi)=\varepsilon g_M (W, \nabla_{V}\xi) =g(W,\phi V)=0$. Thus we have
\begin{align*}
g_M (\nabla_{V}W, X)=&g_M (\phi \nabla_{V}W, \phi X)-\eta(\nabla_{V}W)\eta(X) \\
=& g_M (\phi \nabla_{V}W, \phi X)\\
=& g_M (\nabla_{V} \phi W, \phi X)-g_M \big( (\nabla_{V}\phi)W,\phi X \big)\\
=&-g_M(\phi W , \nabla_{V}\phi X)
\end{align*}
Since F is a semi-Riemannian submersion, we have
\begin{align*}
g_M (\nabla_{V}W, X)=-g_N(F_* \phi W, F_*\nabla_{V} \phi X )=g_N \big( F_*\phi W,(\nabla F_*)(V,\phi X) \big),
\end{align*}
which proves $(a) \Leftrightarrow (b)$.
\\
By direct calculation, we derive
\begin{align*}
g_N \big( F_*\phi W,(\nabla F_*)(V,\phi X) \big)=&-g_M(\phi W , \nabla_{V}\phi X)\\
=& -g_M(\phi W, \nabla_{V}BX+\nabla_{V}CX)\\
=& -g_M(\phi W, \nabla_{V}BX+[V,CX]+\nabla_{CX}V)
\end{align*}
Since $ [V,CX] \in \Gamma(\ker F_*) $, from \eqref{2.5} and \eqref{2.7}, we obtain
\begin{align*}
g_N \big( F_*\phi W,(\nabla F_*)(V,\phi X) \big)=-g_M(\phi W , \T_{V}BX+\mathcal{A}_{CX}V ),
\end{align*}
which proves $ (b) \Longleftrightarrow (c) $.
\end{proof}
\begin{cor}\label{cor:3.9}
Let F be an anti-invariant semi-Riemannian submersion from a (para) Lorentzian Sasakian manifold $(M,g_M,\phi ,\xi,\eta)$ onto a Lorentzian manifold $ (N,g_N) $ with $(\ker F_*)^\perp=\phi(\ker F_*)\oplus\Span\{\xi\}$. Then the following are equivalent.
\begin{enumerate}
\item[(a)]
$ \ker F_* $ defines a totally geodesic foliation on M.
\item[(b)]
$(\nabla F_*)(V,\phi X) =0$
\item[(c)]
$ \T_{V}\phi W=0$, for $X \in \Gamma \big((\ker F_*)^{\perp}\big)$ and $ V,W \in \Gamma(\ker F_*) $.
\end{enumerate}
\end{cor}

The proof of the following two theorems is exactly the same with Theorem 3.10 and Theorem 3.11 in \cite{Lee:AntInvariantRiemannianSubmersionsAlmostContact} for Riemannian case. Therefore we omit them  here.
\begin{theorem}\label{theorem:3.10}
Let F be an anti-invariant semi-Riemannian submersion from a (para) Lorentzian Sasakian manifold $ (M, g_M, \phi, \xi , \eta) $ onto a Lorentzian manifold $ (N,g_N) $ with $ (\ker F_*)^{\perp} =\phi (\ker F_*) \oplus \Span\{\xi\}$.
Then F is a totally geodesic map if and only if
\begin{align}
\T_{V}\phi W=0 \hspace{2mm} V,W \in \Gamma(\ker F_*) \label{44-3.10}
\end{align}
and
\begin{align}
\mathcal{A}_{X}\phi W=0 \hspace{2mm} X \in \Gamma\big((\ker F_*)^{\perp}\big) \label{45-3.10}
\end{align}
\end{theorem}
\begin{theorem}\label{theorem:3.11}
Let F be an anti-invariant semi-Riemannian submersion from a (para) Lorentzian Sasakian manifold $ (M, g_M, \phi, \xi , \eta) $ onto a Lorentzian manifold $ (N,g_N) $ with $ (\ker F_*)^{\perp} =\phi (\ker F_*) \oplus\Span\{\xi\}$.
Then F is a harmonic map if and only if $\tr(\phi \T_{V})=0$ for $ V \in \Gamma(\ker F_*) $.
\end{theorem}
In the following, we obtain decomposition theorems for an anti-invariant semi-Riemannian submersion from a (para)Lorentzian Sasakian manifold onto a Lorentzian manifold. By using results in subsection \ref{decomposition} and Theorems \ref{th:3.4},  \ref{theorem:3.6} and \ref{theorem:3.8}, we have the following theorem
\begin{theorem}\label{theorem:4.2}
Let F be an anti-invariant semi-Riemannian submersion from a (para) Lorentzian Sasakian manifold $ (M, g_M, \phi, \xi , \eta) $ onto a Lorentzian manifold $ (N,g_N) $.
Then M is a locally product manifold if and only if
\begin{align*}
g_N \big( (\nabla F_*)(Y,B X),F_* \phi V \big)=g_M(CY , \phi \mathcal{A}_{X}V)+\varepsilon \eta(Y)g_M(X,\phi V)
\end{align*}
and
\begin{align*}
g_N \big( (\nabla F_*)(V,\phi X),F_* \phi W \big)=0
\end{align*}
for $ X,Y \in \Gamma\big((\ker F_*)^{\perp}\big) $ and $ V,W \in \Gamma(\ker F_*) $.
\end{theorem}

\begin{theorem}\label{theorem:4.4}
Let F be an anti-invariant semi-Riemannian submersion from a (para) Lorentzian Sasakian manifold $ (M, g_M, \phi, \xi , \eta) $ onto a Lorentzian manifold $ (N,g_N) $ with $ (\ker F_*)^{\perp} =\phi (\ker F_*) \oplus\Span\{\xi\}$.
Then M is locally twisted product manifold of the form $ M_{(\ker F_*)^{\perp}} \times_{f} M_{\ker F_*} $ if and only if
 \begin{align*}
\T_{V}\phi X=-g_M (X,\T_{V}V)||V||^{-2} \phi V
\end{align*}
and
\begin{align*}
\mathcal{A}_X \phi Y=\eta(Y) X
\end{align*}
for $ X,Y \in \Gamma\big((\ker F_*)^{\perp}\big) $ and $ V,W \in \Gamma(\ker F_*) $, where $ M_{(\ker F_*)^{\perp}}$ and  $M_{\ker F_*} $ are integral manifolds of the distributions $ (\ker F_*)^{\perp} $ and $ \ker F_* $
\end{theorem}
\begin{theorem}\label{theorem:4.4}
Let $ (M, g_M, \phi, \xi , \eta) $ a (para) Lorentzian Sasakian manifold and $ (N,g_N) $ be a Lorentzian manifold. Then there does not exist an anti-invariant semi-Riemannian submersion from M to N with $ (\ker F_*)^{\perp} =\phi (\ker F_*) \oplus\Span\{\xi\}$ such that M is a locally proper twisted product manifold of the form $ M_{(\ker F_*)^{\perp}} \times_{f} M_{\ker F_*} $.
\end{theorem}


\begin{thebibliography}{10}
\providecommand{\url}[1]{{#1}}
\providecommand{\urlprefix}{URL }
\expandafter\ifx\csname urlstyle\endcsname\relax
  \providecommand{\doi}[1]{DOI~\discretionary{}{}{}#1}\else
  \providecommand{\doi}{DOI~\discretionary{}{}{}\begingroup
  \urlstyle{rm}\Url}\fi

\bibitem{Alegre:Slantsubmanifolds}
Alegre, P.: Slant submanifolds of {L}orentzian {S}asakian and para {S}asakian
  manifolds.
\newblock Taiwanese J. Math. \textbf{17}(3), 897--910 (2013).
\newblock \doi{10.11650/tjm.17.2013.2427}.
\newblock \urlprefix\url{http://dx.doi.org/10.11650/tjm.17.2013.2427}

\bibitem{blair.Ri:book}
Blair, D.E.: Riemannian geometry of contact and symplectic manifolds,
  \emph{Progress in Mathematics}, vol. 203.
\newblock Birkh\"auser Boston, Inc., Boston, MA (2002).
\newblock \doi{10.1007/978-1-4757-3604-5}.
\newblock \urlprefix\url{http://dx.doi.org/10.1007/978-1-4757-3604-5}

\bibitem{chen:2011pseudo}
Chen, B.Y.: Pseudo-Riemannian Geometry, $\delta$-Invariants and Applications.
\newblock World Scientific Publishing Company (2011)

\bibitem{Eells}
Eells Jr., J., Sampson, J.H.: Harmonic mappings of {R}iemannian manifolds.
\newblock Amer. J. Math. \textbf{86}, 109--160 (1964)

\bibitem{Erken:AntiInvariantRiemanniansubmersionsfromSasakianmanifolds}
Erken, I.K., Murathan, C.: Anti-invariant {R}iemannian submersions from
  {Sasakian} manifolds.
\newblock arXiv preprint arXiv:1302.4906  (2013)

\bibitem{Erken:SlantRiemanniansubmersionsfromSasakianmanifolds2016}
Erken, I.K., Murathan, C.: Slant riemannian submersions from sasakian
  manifolds.
\newblock Arab Journal of Mathematical Sciences \textbf{22}(2), 250--264
  (2016).
\newblock \doi{10.1016/j.ajmsc.2015.12.002}

\bibitem{Fuglede:Harmonicmorphismsbetween}
Fuglede, B.: Harmonic morphisms between semi-{R}iemannian manifolds.
\newblock Ann. Acad. Sci. Fenn. Math. \textbf{21}(1), 31--50 (1996)

\bibitem{Gray:PseudoRiemannianAlmostProductManifoldsSubmersions}
Gray, A.: Pseudo-{R}iemannian almost product manifolds and submersions.
\newblock Indiana Univ. Math. J. \textbf{16}, 715--737 (1967).
\newblock \doi{10.1512/iumj.1967.16.16047}

\bibitem{Gunduzlap:SlantsubmersionsfromLorentzian}
G{\"u}nd{\"u}zalp, Y.: Slant submersions from {L}orentzian almost paracontact
  manifolds.
\newblock Gulf J. Math. \textbf{3}(1), 18--28 (2015).
\newblock
  \urlprefix\url{http://gjom.org/wp-content/uploads/2015/02/v7-2-106.pdf}

\bibitem{GUNDUZALP:paracontactparacomplexsemiRiemanniansubmersion}
G\"und\"uzalp, Y., Sahin, B.: Para-contact para-complex semi-{R}iemannian
  submersions.
\newblock Bull. Malays. Math. Sci. Soc. (2) \textbf{37}(1), 139--152 (2014)

\bibitem{gunduzlapShahin:ParacontactSemiRiemannianSubmersions}
G\"und\"uzalp, Y.l., \c~Sahin, B.: Paracontact semi-{R}iemannian submersions.
\newblock Turkish J. Math. \textbf{37}(1), 114--128 (2013).
\newblock \doi{10.3906/mat-1103-10}

\bibitem{Kaneyuki:Almostparacontact}
Kaneyuki, S., Williams, F.L.: Almost paracontact and parahodge structures on
  manifolds.
\newblock Nagoya Math. J. \textbf{99}, 173--187 (1985).
\newblock \urlprefix\url{http://projecteuclid.org/euclid.nmj/1118787873}

\bibitem{Lee:AntInvariantRiemannianSubmersionsAlmostContact}
Lee, J.W.: Anti-invariant {$\xi^\perp$}-{R}iemannian submersions from almost
  contact manifolds.
\newblock Hacet. J. Math. Stat. \textbf{42}(3), 231--241 (2013)

\bibitem{Magid:SubmersionsFromSntiDeSitterspace}
Magid, M.A.: Submersions from anti-de {S}itter space with totally geodesic
  fibers.
\newblock J. Differential Geom. \textbf{16}(2), 323--331 (1981).
\newblock \urlprefix\url{http://projecteuclid.org/euclid.jdg/1214436107}

\bibitem{Falcitelli:RiemannianSubmersionsandRelatedTopics}
Maria~Falcitelli Anna Maria~Pastore, S.I.: Riemannian Submersions and Related
  Topics.
\newblock World Scientific Publishing Company (2004).
\newblock
  \urlprefix\url{http://gen.lib.rus.ec/book/index.php?md5=44E2548F7A4EFF6C3C7C%
739FD57E02FB}

\bibitem{Murathan:AntiInvariantRiemannianSubmersionFromCosymplectic}
Murathan, C., K{\"u}peli~Erken, I.: Anti-invariant {R}iemannian submersions
  from cosymplectic manifolds onto {R}iemannian manifolds.
\newblock Filomat \textbf{29}(7), 1429--1444 (2015).
\newblock \doi{10.2298/FIL1507429M}.
\newblock \urlprefix\url{http://dx.doi.org/10.2298/FIL1507429M}

\bibitem{Oneill:Thefundamentalequationsofsubmersion}
O'Neill, B.: The fundamental equations of a submersion.
\newblock Michigan Math. J. \textbf{13}, 459--469 (1966)

\bibitem{oneil:book}
O'Neill, B.: Semi-{R}iemannian geometry, \emph{Pure and Applied Mathematics},
  vol. 103.
\newblock Academic Press Inc. [Harcourt Brace Jovanovich Publishers], New York
  (1983).
\newblock With applications to relativity

\bibitem{ponge:twisted.product}
Ponge, R., Reckziegel, H.: Twisted products in pseudo-{R}iemannian geometry.
\newblock Geom. Dedicata \textbf{48}(1), 15--25 (1993).
\newblock \doi{10.1007/BF01265674}.
\newblock \urlprefix\url{http://dx.doi.org/10.1007/BF01265674}

\bibitem{shahin:AntiInvariantRiemannianSubmersionsFromAlmostHermitianManifolds}
Sahin, B.: Anti-invariant {R}iemannian submersions from almost {H}ermitian
  manifolds.
\newblock Cent. Eur. J. Math. \textbf{8}(3), 437--447 (2010).
\newblock \doi{10.2478/s11533-010-0023-6}.
\newblock \urlprefix\url{http://dx.doi.org/10.2478/s11533-010-0023-6}

\bibitem{Shahin:SlantSubmersionHermitian}
{\d{S}}ahin, B.: Slant submersions from almost {H}ermitian manifolds.
\newblock Bull. Math. Soc. Sci. Math. Roumanie (N.S.) \textbf{54(102)}(1),
  93--105 (2011)

\end{thebibliography}
\end{document}